\documentclass{amsart}

\usepackage{amssymb}
\usepackage{graphicx}

\newtheorem{thm}{Theorem}[section]
\newtheorem{prop}[thm]{Proposition}
\newtheorem{lem}[thm]{Lemma}
\newtheorem{cor}[thm]{Corollary}

\newtheorem{conj}[thm]{Conjecture}

\theoremstyle{definition}
\newtheorem{definition}[thm]{Definition}

\newtheorem{obs}[thm]{Observation}
\newtheorem{rmk}[thm]{Remark}
\newtheorem{conv}[thm]{Convention}

\numberwithin{equation}{section}

\newcommand{\D}{\mathcal{D}}
\newcommand{\A}{\mathcal{A}}
\newcommand{\B}{\mathcal{B}}
\newcommand{\C}{\mathcal{C}}
\newcommand{\E}{\mathcal{E}}
\newcommand{\G}{\mathcal{G}}
\newcommand{\F}{\mathcal{F}}
\newcommand{\V}{\mathcal{V}}
\newcommand{\X}{\mathcal{X}}
\newcommand{\Y}{\mathcal{Y}}
\newcommand{\W}{\mathcal{W}}
\newcommand{\Hf}{\mathcal{H}}
\newcommand{\Q}{\mathcal{Q}}

\begin{document}


\title{A lower bound on tunnel number degeneration}


\author{Trent Schirmer}
\address{Department of Mathematics, Oklahoma State University, 
Stillwater, OK 74078}
\email{trent.schirmer@okstate.edu}
\urladdr{www.trentschirmer.com} 



\begin{abstract}
We prove a theorem which bounds Heegaard genus from below under special kinds of toroidal amalgamations of $3$-manifolds.  As a consequence, we conclude $t(K_1\# K_2)\geq \max\{t(K_1),t(K_2)\}$ for any pair of knots $K_1,K_2\subset S^3$, where $t(K)$ denotes the tunnel number of $K$.
\end{abstract}

\maketitle

The tunnel number $t(K)$ of a knot $K\subset S^3$ can be defined by the equation $t(K)+1=g(S^3-\eta(K))$, where $g(\cdot )$ denotes Heegaard genus and $\eta(K)$ is an open regular neighborhood of $K$.  In more intuitive terms, the tunnel number of a knot is the minimal number of ``tunnels'' that must be drilled through $S^3-\eta(K)$ in order to make the resulting manifold a handlebody.

The behavior of $t(K)$ under the operation of connected sum has been studied extensively.  It is not difficult to see that $t(K_1\# K_2)\leq t(K_1)+t(K_2)+1$, although it takes some work to find examples where equality is achieved in this bound \cite{kobri}, \cite{morrub}, \cite{mor2}.  It is also known that $t(K_1\# K_2)<t(K_1)+t(K_2)$ for some pairs of knots \cite{mor3}, \cite{nog}, \cite{schir}, and that the degeneration $t(K_1)+t(K_2)-t(K_1\# K_2)$ can be arbitrarily large \cite{kob}, \cite{schir}.

Perhaps most difficult is the task of finding lower bounds on $t(K_1\# K_2)$.  It is known that $t(K_1\# K_2)\geq 2$ for any pair of non-trivial knots in $S^3$ \cite{norwood}, and more generally that $t(K_1\# \cdots \# K_n)\geq n$ \cite{schulschar2}.  In the case that $K_1$ and $K_2$ are small, it is known that $t(K_1\# K_2)=t(K_1)+t(K_2)$ \cite{morimoto-schultens} (the hard part of this is to show that $t(K_1\# K_2)\geq t(K_1)+t(K_2)$), in fact this equation holds even under the assumption that $K_1$ and $K_2$ are meridionally small \cite{kobayashi-rieck}.  Another impressive lower bound is $\frac{t(K_1\# K_2)}{t(K_1)+t(K_2)}\geq \frac{2}{5}$ \cite{schulschar}, which holds for any pair of non-trivial knots in $S^3$ (in fact a more general analogue involving iterated connected sums is derived in \cite{schulschar}).

In this paper, we prove that $t(K_1\# K_2)\geq \max\{t(K_1),t(K_2)\}$ for any pair of knots $K_1,K_2\subset S^3$.  This bound was previously unknown, although there are many examples which show it to be best possible, including those of \cite{mor3} and \cite{nog}.  The proof is centered around the rather involved construction of so-called {\em doppelg\"{a}nger} surfaces.

In rough outline, the strategy of our proof is as follows.  Suppose without loss of generality that $\max\{t(K_1),t(K_2)\}=t(K_2)$ and let $K_1\# K_2$ be realized via the satellite construction with $K_1$ as the companion and $K_2$ as the pattern.  This means that $K_1\# K_2$ lies in $V=\overline{\eta(K_1)}$, and if $h:V\rightarrow S^3$ is the standard unknotted embedding of the solid torus $V$, then $h(K_1\# K_2)=K_2$.  If $\G$ is a thin generalized Heegaard surface of $S^3-\eta(K_1\# K_2)$, then $\G$ can be isotoped to intersect $\overline{S^3-V}$ in a particularly nice way.  Taking into account certain information contained in the intersection $\G\cap(\overline{S^3-V})$, we can then construct the previously mentioned {\em doppelg\"{a}nger surface} $\Q$ inside of a solid torus $W=\overline{S^3-h(V)}$ which, in certain important respects, imitates the placement of the surface $\G\cap (\overline{S^3-V})$ in $(\overline{S^3-V})$.  As a result, $\Q\cup h(\G\cap V)$ forms a generalized Heegaard surface of $\overline{S^3-h(V)}=S^3-\eta(K_2)$ which amalgamates to a surface of lower genus than the amalgamation of $\G$.  This yields the desired lower bound. 

In Section 1 we introduce the notion of a {\em generalized compression body}, which form the basic pieces of $\overline{S^3-V}-\G$ and $W-\Q$, and prove a series of essential cutting and pasting lemmas about them.  Section 2 then describes and works out the basic topology of so-called {\em spoke graphs} and {\em spoke surfaces}, which form the building blocks of the doppelg\"{a}nger surface $\Q$. Section 3 then constructs $\Q$ in detail and proves that it has the desired properties, culminating in the main technical result of the paper, Theorem 3.23.  In Section 4 the bound $t(K_1\# K_2)\geq \max\{t(K_1),t(K_2)\}$ is proved (Theorem 4.1), and some topics related to it are briefly discussed.

Throughout this paper, $N(Y,X)$ denotes a {\em closed} regular neighborhood of $Y$ in $X$, $E(Y,X)=\overline{X-N(Y,X)}$, and $Fr(Y,X)=N(Y,X)\cap E(Y,X)$, or equivalently, $Fr(Y,X)=\overline{\partial N(Y)-\partial X}$. We assume throughout that $N(Y,X)$ behaves well with respect to intersection, so that $N(Y_1,X)\cap N(Y_2,X)=N(Y_1\cap Y_2, X)$.  If $\X$ is a topological space, $|\X|$ denotes the number of components of $\X$.  An embedding of manifolds $f:X\rightarrow Y$ is said to be {\em proper} if $f$ is transverse to $\partial Y$ and $f(\partial X)\subset \partial Y$. A {\em proper isotopy} is a homotopy through proper embeddings (note that this does {\em not} imply that the boundary remains fixed).  As an informal aid to the reader, topological spaces which are allowed to have multiple connected components will usually be denoted in calligraphic font, e.g. $\A$, $\X$, $\Y$, whereas connected topological spaces will usually be denoted in standard font, e.g. $A$, $X$, $Y$.

\section{Generalized compression bodies}

\begin{definition}

Let $\F$ be a compact orientable surface, and let $\V=(\F\times I)\cup(2-handles)\cup (3-handles)$, where the $2$-handles are attached along essential, non-boundary parallel curves in $\F\times\{0\}$, and $3$-handles are attached along all spherical components of $\F\times I\cup (2-handles)$ which are disjoint from $\F\times \{1\}$.  Then $\V$ is called a {\em generalized compression body over $\F$}, or simply a {\em generalized compression body}.  Let $\partial_+\V=\F\times\{1\}$, $\partial_v\V=(\partial \F)\times I$, and $\partial_-\V=\overline{\partial \V-(\partial_+\V\cup \partial_v\V)}$.  If $\V$ is connected and $\partial_v\V=\emptyset$, $\V$ is a {\em compression body}.  If $\V$ is connected and $\partial_-\V=\emptyset$, $\V$ is a {\em handlebody}.

\end{definition}


\begin{obs}

Suppose $\V$ is a generalized compression body, $A_1,A_2$ are disjoint components of $\partial_v\V$, and $h:A_1\rightarrow A_2$ is an orientation reversing homeomorphism which preserves $\partial_+\V$.  Then $\V/h$ is a generalized compression body over $(\partial_+\V)/h$.

\end{obs}



\begin{obs}

If $\V$ is a generalized compression body and $\W$ is obtained by compressing $\V$ along a properly embedded disk $D$ such that $\partial D\subset \partial_+\V$, then $\W$ is again a generalized compression body.  Going the other way, if $\W$ is obtained from $\V$ by attaching an oriented $1$-handle along $\partial_+\V$, then $\W$ is a generalized compression body.

\end{obs}

\begin{figure}
    \centering
    \includegraphics[width=0.8\textwidth]{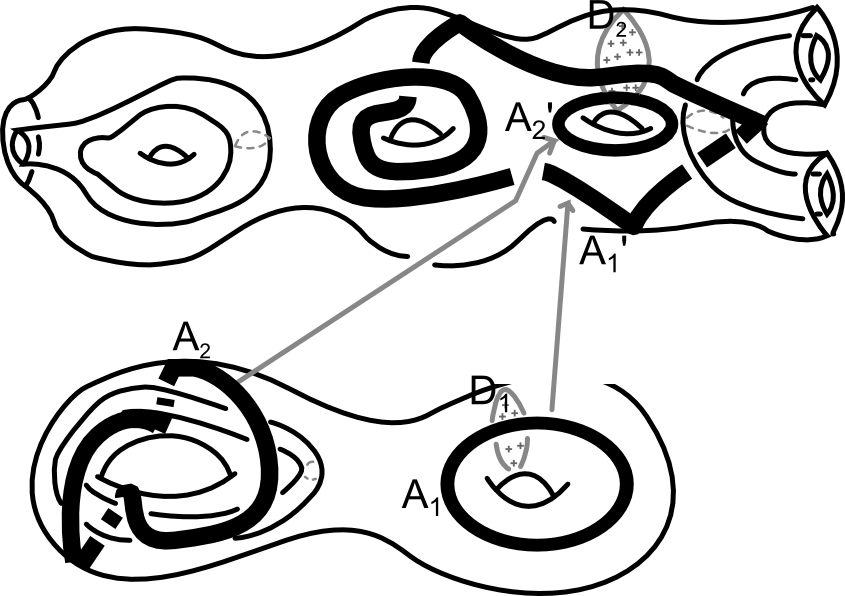}
    \caption{Primitive disk set for a paired union of annuli}
\end{figure}

\begin{definition}

Let $\V$ be a generalized compression body and let $\A=A_1\cup\cdots \cup A_n$ and $\A'=A_1'\cup \cdots \cup A_n'$ be disjoint unions of annuli embedded in $\partial_+\V$ satisfying $\A\cap \A'=\emptyset$.  Let $\D=D_1\cup \cdots \cup D_n$ be a disjoint union of compressing disks for $\V$ such that $\partial\D\subset \partial_+\V$.  If $D_i\cap (A_i\cup A_i')$ consists of a single spanning arc in one of $A_i$ or $A_i'$ for all $1\leq i\leq n$, and $D_i\cap (A_j\cup A_j')=\emptyset$ if $j>i$, then $\D$ is said to be a {\em primitive disk set} for $\A\cup \A'$, and the component of $A_i\cup A_i'$ which meets $D_i$ is said to be {\em dual} to $D_i$.  The above orderings of the components of $\A\cup \A'$ and $\D$ will be called the {\em primitive ordering} associated with $\D$.  See Figure 1.

\end{definition}

\begin{rmk}

The choice of primitive ordering is essential to Definition 1.4, and some fixed choice is always assumed to be present when we are dealing with a primitive disk set $\D$ for a paired union $\A\cup \A'$ of annuli.  For the most part, however, the primitive ordering will only be specified explicitly when necessary.

\end{rmk}

\begin{prop}

Let $\V$ be a generalized compression body and let $\A=A_1\cup\cdots \cup A_n$ and $\A'=A_1'\cup \cdots \cup A_n'$ be disjoint unions of annuli embedded in $\partial_+\V$ satisfying $\A\cap \A'=\emptyset$.  Let $h:\A\rightarrow \A'$ be an orientation reversing homeomorphism such that $h(A_i)=A_i'$ for all $1\leq i\leq n$, and suppose that $\A\cup \A'$ admits a primitive disk set.  Then $\V/h$ is a generalized compression body over $(\partial_+\V)/h$.

\end{prop}

\begin{proof}

We proceed by induction on $n$.  In the base case $n=0$ there is nothing to prove. If $n>0$, suppose without loss of generality that $A_n$ is dual to $D_n$ (the argument is the same if $A_n'$ is dual $D_n$).  By Observation 1.3, $E(D_n,\V)$ is again a generalized compression body, and the result of reattaching $N(D_n)$ to $E(D_n,\V)$ via the map $h|_{N(D_n\cap A_n,A_n)}$ again results in a generalized compression body $\V'$, since this amounts to trivially attaching a ball to $\partial_+ E(D_n,\V)$ along a disk on its boundary.  Since $\V/h|_{A_n}$ is obtained from $\V'$ via a $1$-handle attachment along $\partial_+\V'$, Observation 1.3 tells us that $\V/h|_{A_n}$ is a generalized compression body.  Since $D_1\cup \cdots\cup D_{n-1}$ was disjoint from $A_n\cup A_n'$ and $D_n$, it remains a primitive disk set for $A_1\cup\cdots \cup A_{n-1}$ and $A_1'\cup\cdots \cup A_{n-1}'$ in $\V/h|_{A_n}$, and the desired conclusion follows by induction. 

\end{proof}

\begin{prop}

Suppose $\V$ is a generalized compression body, let $\A=A_1\cup\cdots\cup A_n$ be a disjoint union of annuli embedded in $\partial_+\V$, and let $\D=D_1\cup\cdots \cup D_n$ be a disjoint union of disks properly embedded in $\V$ such that $\partial \D\subset\partial_+\V$.  If $D_i\cap A_i$ consists of a single spanning arc in $A_i$ for all $1\leq i\leq n$, and $D_i\cap A_j=\emptyset$ whenever $i<j$, then manifold $\W$ obtained by attaching 2-handles along $\A$ is again a generalized compression body.

\end{prop}

\begin{proof}

The proposition is well known in the case the $\V$ is a compression body.  The proof in the general case here is essentially the same as that of Proposition 1.6.

\end{proof}

\begin{definition}

An annulus $A$ properly embedded in a generalized compression body $V$ is said to be {\em spanning} if one component of $\partial A$ lies on $\partial_+V$ and the other lies on $\partial_-V$.  $A$ is said to be {\em horizontal} if $\partial A\subset \partial_+V$.

\end{definition}


\begin{prop}

If $F$ is a compact surface and $\A$ is a disjoint union of incompressible spanning annuli embedded in $F\times I$, then $F\times I$ can be re-parameterized so that $\A=\C\times I$ for some disjoint union of essential simple closed curves $\C\subset F$.

\end{prop}

\begin{proof}

This is a well known fact which often appears in the literature, so the following proof is merely a sketch.  If $\C=\A\cap (F\times\{0\})$, then $\C\times I$ is another union of spanning annuli $\A'$ such that $\A'\cap (F\times \{1\})$ is isotopic to $\A\cap (F\times\{1\})$ in $F\times \{1\}$ (this follows from the $\pi_1$-injectivity of $\A$ and $\A'$).  This allows $\A'$ to be properly isotoped so that $\partial \A'$ and $\partial \A$ are parallel and disjoint in $\F\times \{0,1\}$.  Since $\A$ and $\A'$ are both incompressible, and $F\times I$ is irreducible, any simple closed curves in $\A'\cap \A$ which are trivial in either of $\A'$ or $\A$ can be eliminated via further isotopy of $\A'$ using standard inner-most disk arguments.  Again, since each component of $\A\cup \A'$ is $\pi_1$-injective, any remaining components of $\A\cap \A'$ must come from pairs of annuli $A,A'$ with isotopic boundaries on $\F\times\{0,1\}$.  Thus these components of $\A\cap \A'$ can also be removed, uppermost ones first, using the fact that any incompressible horizontal annulus in $F\times I$ with parallel boundary components will co-bound a solid torus with an annulus on $F\times\{1\}$.  Once $\A'$ has been made disjoint from $\A$, the components of $\A\cup \A'$ will co-bound solid tori with annuli in $F\times\{0,1\}$, so that $\A'$ can finally be isotoped onto $\A$.  Extending this proper isotopy of $\A'$ to an ambient isotopy of $\F\times I$ yields the desired reparameterization.

\end{proof}

\begin{prop}

Let $\V$ be a compression body and let $\A=\A_s\cup \A_h$ be a disjoint union of incompressible annuli properly embedded in $V$, so that every component of $\A_s$ is spanning and every component of $\A_h$ is horizontal.  Then $E(\A,\V)$ is a generalized compression body $\V'$ such that $\partial_v(\V')=Fr(\A_s)$ and $Fr(\A_h)\subset \partial_+\V'$.  Moreover, for an appropriate ordering of the components of $\A_h=A_1\cup\cdots\cup A_n$, if we set $Fr(A_i)=A_i'\cup A_i''$, $\A'=A_1'\cup\cdots \cup A_n'$, and $\A''=A_1''\cup\cdots \cup A_n''$, then the collection $\A'\cup \A''$ admits a primitive disk set in $\V'$.

\end{prop}

\begin{proof}

Lemma 2 of \cite{schul} tells us that $\W=E(\A_h,\V)$ is a union of compression bodies.  The annuli $\A_s$ remain spanning in $\W$, thus are disjoint from some union of disks $\E$ properly embedded in $\W$ such that $E(\E,\W)\cong \partial_-\W\times I$ (see, e.g., Lemma 3.15 of \cite{saitoschulschar}).  By Proposition 1.9, we may assume that $E(\E,\W)$ has been parameterized so that $\A_s$ has the form $\mathcal{C}\times I$ in $\partial_-\W\times I$, where $\mathcal{C}\subset \partial_-\W$ is a disjoint union of simple closed curves.  Thus $\V'=E(\A_s,\W)=E(\A,\V)$ is a generalized compression body satisfying $\partial_v\V'=Fr(\A_s)$, as claimed.

We prove the second part by induction on $n$.  There is nothing to prove if $n=0$.  If $n>0$, then Lemma 9 of \cite{Bonahon-Otal} implies that $\A_h$ is boundary-compressible in $\V$ via some disk $D_1$.  We may choose $D_1$ so that it is disjoint from $\A_s$, and we assign the label $A_1$ to the component of $\A_h$ which has been boundary-compressed by $D_1$. By the first part of this lemma, $E(\A_s\cup A_1,\V)$ is a generalized compression body and so by induction $Fr(\A_h-A_1)$ admits a primitive disk set $\D'$ in $\V'=E(\A,\V)$ with respect to an appropriate choice of numbering for $\A_h-A_1$, with $A_2$ as the lowest indexed annulus.  We may also choose (appropriately indexed) $\D'$ so that $\D'\cap D_1=\emptyset$, and since $D_1\cap (\A_h-A_1)=\emptyset$, it follows that $\D=\D'\cup (D_1\cap \V')$ is a primitive disk set for $Fr(\A_h)$ in $\V'$, as required.

\end{proof}

\begin{definition}

A graph $\X$ embedded in a manifold $M$ is said to be {\em properly embedded} if $\X$ is transverse to $M$ and $\X\cap\partial M$ is a union of elements from the set of univalent vertices of $\X$.  $\X$ is said to be {\em unknotted} if it can be isotoped into $\partial M$ via an isotopy $\Phi:\X\times I\rightarrow M$ such that the function $\Phi|_{\{t\}\times X}$ is a proper embedding for all $0\leq t<1$.

\end{definition}

\begin{obs}

If $X$ is an unknotted tree properly embedded in the $3$-ball $B$, and $F$ is the surface which results from removing an open collar of $\partial \overline{\partial B-N(X)}$ from $\overline{\partial B-N(X)}$, then there is homeomorphism $h:F\times I\cong E(X,B)$ such that $h(F\times\{1\})=F$ and $h(F\times\{0\})=Fr(X,B)$. 

\end{obs}

\begin{definition}
Let $X$ be a graph embedded in a handlebody $V$ such that $E(X,V)\cong\partial V\times I$.  Then $X$ is called a {\em spine} of $V$.
\end{definition}

\begin{prop}

Suppose $X$ is a graph embedded in the interior of a handlebody $V$, and that there is a disjoint union of compressing disks $\D$ properly embedded in $V$ so that the following are true:

\begin{enumerate}

\item $E(\D,V)$ is a union of balls,
\item $X\cap B$ is a properly embedded, unknotted tree for each component $B$ of $E(\D,V)$,
\item For each component $D$ of $\D$, $D\cap X$ is a single point on an edge of $X$.

\end{enumerate}

Then $X$ is a spine of $V$.

\end{prop}

\begin{proof}

For each component $B$ of $E(\D,V)$, let $\E\subset \partial B$ be the set of ``scars'' left behind by cutting along $\D$, i.e. $\E=N(\D,V)\cap B$, and let $X'=X\cap B$.  Then hypotheses (2) and (3) give us a homeomorphism $h:F\times I\rightarrow E(X',B)$ as per Observation 1.12, where we may take $F=\overline{\partial B-\E}$.  These homeomorphisms can then be pasted together to form a homeomorphism between $E(X,V)$ and $\partial V\times I$.

\end{proof}

\begin{prop}

Suppose $V$ is a handlebody with spine $X$, and that $Y$ is a subgraph of $X$ without simply connected components.  Then $E(Y,V)$ is a compression body.

\end{prop}

\begin{proof}
$E(Y,V)$ is obtained from $E(X,V)\cong \partial V\times I$ by attaching $2$-handles and $3$-handles along $\partial V\times\{0\}$.
\end{proof}

\begin{definition}

A {\em Heegaard splitting} $(V,W,G)$ of a manifold $M$ is a decomposition $M=V\cup W$ where each of $V,W$ is a compression body and $G=\partial_+ V=\partial_+ W=V\cap W$.  $G$ is called a {\em Heegaard surface}.  The {\em Heegaard genus} $g(M)$ of a manifold is the minimal genus of a Heegaard surface for $M$.

\end{definition}

\begin{definition}
A {\em generalized Heegaard splitting} $((V_1,W_1,G_1),\cdots , (V_n,W_n,G_n))$ of a manifold $M$ is a decomposition $M=M_1\cup\cdots \cup M_n$ such that $(V_i,W_i,G_i)$ forms a Heegaard splitting for the submanifold $M_i$, $M_i\cap M_{i+1}=\partial_-W_i=\partial_-V_{i+1}=F_i$ for $1\leq i<n$, and $M_i\cap M_j=\emptyset$ whenever else $i\neq j$ (we include the case $n=1$ corresponding to standard Heegaard splittings).  The surface $\G=G_1\cup\cdots\cup G_n\cup F_1\cup\cdots\cup F_{n-1}$ is called a {\em generalized Heegaard surface}.  Given a generalized Heegaard surface $\G$, we let $\G_+=G_1\cup\cdots\cup G_n$ denote the {\em thick surfaces} of $\G$, and $\G_-=\G-\G_+$ denote the {\em thin surfaces} of $\G$.

\end{definition}

\begin{definition}
A Heegaard splitting $(V,W,G)$ is said to be:

\begin{itemize}
\item {\em Stabilized} if there exist compressing disks $D\subset V$, $D'\subset W$ such that $D\cap D'$ is a single point.
\item {\em Reducible} if there exist compressing disks $D\subset V$, $D'\subset W$ such that $\partial D =\partial D'$.
\item {\em Weakly reducible} if it is not stabilized or reducible, but there are compressing disks $D\subset V$, $D'\subset W$ such that $D\cap D' =\emptyset$
\item {\em Strongly irreducible} if it is not stabilized and, for all compressing disks $D\subset V$, $D'\subset W$, $D\cap D' \neq \emptyset$
\end{itemize}
\end{definition}

\begin{rmk}
There is process of {\em untelescoping} a weakly reducible Heegaard splitting $(V,W,G)$ whereby it is changed into a generalized Heegaard splitting $((V_1,W_1,G_1),\cdots , (V_n,W_n,G_n))$ satisfying $g(G)=\Sigma g(G_i) -\Sigma g(F_i)$.  Conversely, given a generalized Heegaard splitting $((V_1,W_1,G_1),\ldots ,(V_n,W_n,G_n))$ for $M$, one can always use the process of {\em amalgamation} to change it into a standard Heegaard splitting $(V,W,G)$ of $M$ satisfying the same equation.  The interested reader is referred to \cite{saitoschulschar} for the details of these processes and a proof of the following lemma.
\end{rmk}

\begin{prop} \cite{scharthomp}
If $(V,W,G)$ is a weakly reducible Heegaard splitting of $M$, then $(V,W,G)$ can be untelescoped to a generalized Heegaard splitting $((V_1,W_1,G_1),$ $\ldots ,(V_n,W_n,G_n))$ such that $(V_i,W_i,G_i)$ is a strongly irreducible splitting of $M_i$ for each $1\leq i \leq n$, and the thin surfaces $F_i$ are incompressible in $M$ for each $1\leq i <n$.  In this case the generalized splitting is said to be {\em fully untelescoped}.  A standard Heegaard splitting that is strongly irreducible will also be considered {\em fully untelescoped}.

\end{prop}

\begin{prop}\cite{schulschar2}
If $\G$ is the union of the thick and thin surfaces of a fully untelescoped Heegaard splitting of $M$, and $T$ is an incompressible surface properly embedded in $M$, then $\G$ can be isotoped so that it meets $T$ only in simple closed curves which are non-trivial in both $\G$ and $T$.

\end{prop}

\section{Spoke surfaces in the solid torus}

\begin{conv}
Throughout this section, we set $W=S^1\times D^2$ and parameterize it using polar coordinates $(\phi,r,\theta)$, $0\leq \phi,\theta\leq 2\pi$, $0\leq r\leq 1$.
\end{conv}

\begin{definition}
Let $X\subset W$ be an embedded graph with one central vertex $x_0$ at $(\phi_0,0,0)$, a finite number of outer vertices $\{x_1,\cdots , x_n\}\subset \{\phi_0\}\times \partial D^2$, one radial edge connecting each outer vertex $x_i$ to the central vertex $x_0$, and one longitunidal edge $l_i=S^1\times \{x_i\}$ connecting $x_i$ to itself, $1\leq i\leq n$.  Then $X$ is said to be a {\em connected spoke graph} in $W$.  A finite, disjoint union of connected spoke graphs $\X$ is simply called a {\em spoke graph}, $Fr(\X)$ is called a {\em spoke surface}, $E(\X,W)$ is a {\em spoke chamber}.
\end{definition}

\begin{definition}
Suppose $X$ is a connected spoke graph with central vertex at $(\phi_0,0,0)$.  Let $\D$ be a disjoint union of disks embedded in $\mathring{W}$ such that for each component $D$ of $\D$,

\begin{itemize}

\item $D\cap X$ is a connected subarc of a radial edge of $X$,

\item $D\cap (\{\phi_0\}\times D^2)=D\cap X$.

\end{itemize}

Then $X\cup \partial\D$ is called a {\em stabilized} spoke graph with stabilizing disk set $\D$, $Fr(X\cup\partial \D)$ is a stabilized spoke surface, and $E(X\cup \partial D,W)$ is a stabilized spoke chamber.  Moreover, $\overline{\partial \D-X}$ is called the set of {\em stabilizing arcs} of $X$.

\end{definition}

\begin{definition}

Let $X$ be a connected stabilized spoke graph with stabilizing disk set $\D$ and central vertex at $(\phi_0,0,0)$.  Let $\A=E(X\cap \partial W,\partial W)$, and let $\E=(\{\phi_0\}\times D^2)\cap V$, where $V=E(X,W)$.  Then the {\em standard disk set} of $X$ is the union of disks $\D_X=\E\cup(\D\cap V)\cup Fr(\E\cup \A,V)$.

\end{definition}

\begin{obs}

The standard disk set of $X$ cuts $E(X,W)$ into a union of balls.  Thus $E(X,W)$ is a handlebody.

\end{obs}

\begin{figure}
    \centering
    \includegraphics[width=0.8\textwidth]{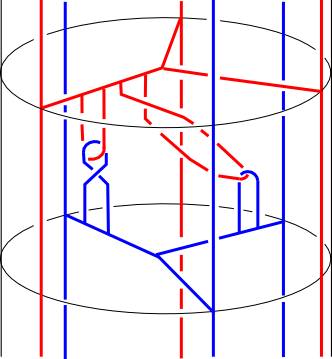}
    \caption{Dual stabilized spoke graphs}
\end{figure}

\begin{definition}

Let $X$ be a connected, stabilized spoke graph with stabilizing disk set $\D$, and whose central vertex has $\phi$-coordinate $\phi_1$.  Let $X'$ be another connected, stabilized spoke graph disjoint from $X$ with stabilizing disk set $\D'$, whose central vertex has $\phi$-coordinate $\phi_2\neq \phi_1$, and suppose the following properties are also satisfied:

\begin{itemize}

\item The set of longitudinal edges of $X'$ is precisely the set of core curves of the annuli $E(X\cap \partial W,\partial W)$.
\item $\D\cap (\{\phi_2\}\times D^2)=\emptyset=\D'\cap (\{\phi_1\}\times D^2)$.
\item Every component of $D$ of $\D$ meets precisely one component $D'$ of $\D'$ in a single arc which has one endpoint on $\partial D-X$, and the other on $\partial D'-X'$, and conversely each component of $\D'$ meets precisely one component of $\D$ in this way.
\item Let $h$ be the projection $S^1\times D^2\rightarrow S^1$. Then $h|_{\mathring{\D}}$ is circle-valued Morse function without singularities, and for every stabilizing arc $\alpha$ of $X$, $h|_\alpha$ is Morse with only one critical point occurring at $\alpha\cap \D'$.  Likewise for the stabilizing disks and arcs of $X'$.

\end{itemize}  

Then $X$ and $X'$ are said to be {\em dual} to one another.  See Figure 2.

\end{definition}

\begin{rmk}

One way to obtain a dual graph $X'$ is to rotate a copy of $X$ slightly in the $\phi$ and $\theta$ directions, and then isotope the stabilizing arcs of $X'$ slightly so that they clasp those of $X$ in one to one fashion.  However, we use the Morse condition on the stabilizing arcs because it allows us complete flexibility in the choice of radial edges of $X'$ at which to base its stabilizing arcs, while still avoiding knottedness.

\end{rmk}

\begin{lem}

Let $X$ and $X'$ be a dual pair of connected spoke surfaces in $W$, and let $B$ be the component of $E(X,W)-\mathring{N}(\D_X)$ which contains the central vertex of $X'$, where $\D_X$ is the standard disk set of $X$ from Definition 2.4.  Then $B$ is a ball, and $X'\cap B$ is an unknotted tree properly embedded in $B$.

\end{lem}

\begin{proof}

It is clear that $B$ is a ball, we show that $X'\cap B$ is unknotted.  We retain the notation of Definition 2.6 throughout.  If $\C$ is the set of stabilizing arcs of $X'$, then the Morse condition on $\C$ ensures that $h|_{\C\cap B}$ is Morse without singularities, and since $h|_{\mathring{\D}}$ is also Morse without singularities, we can slide the endpoints of $\C\cap B$ off of $N(\D)\cap \partial B$ without introducing any further singularities.  The components of $N(\D,W)\cap \partial B$ can then be ``pushed in'' so that $\partial B-\partial W$ is level with respect to $h$, and the arcs of $\C\cap B$ can then be properly isotoped horizontally with respect to $h$ until they are vertical.  After these isotopies, it is clear that $X'\cap B$ is unknotted. 

\end{proof}

\begin{lem}

If $X$ and $X'$ are connected, dual, stabilized spoke graphs, then $E(X\cup X',W)$ is homeomorphic to $Fr(X)\times I$ via a map sending $Fr(X')$ to $Fr(X)\times\{0\}$ and $Fr(X)$ to $Fr(X)\times \{1\}$.

\end{lem}

\begin{proof}

Take the double of $W$ to obtain $S^1\times S^2$, let $X_d$ be the double of of $X$, and let $X_d'$ be the double of $X'$.  Then $E(X_d,S^1\times S^2)$ is a handlebody, since the double $\E$ of the standard disk set $\D_X$ cuts $E(X_d,S^1\times S^2)$ into balls.  Moreover, $E(X_d,S^1\times S^2)$, $\E$, and $X_d'$ satisfy the hypotheses of Proposition 1.14 (Lemma 2.8 handles the only subtle aspect of this).  Thus $X_d'$ is a spine of $E(X_d,S^1\times S^2)$, and we obtain a parameterization $E(X_d\cup X_d',S^1\times S^2)\cong \partial N(X_d)\times I$.  By Proposition 1.9, the spanning annuli $(\partial W)\cap E(X_d\cup X_d',S^1\times S^2)$ can be assumed to be vertical with respect to this parameterization, and the result follows.

\end{proof}

\begin{rmk}

Besides being a steppingstone to Proposition 2.18 below, the significance of Lemma 2.9 is that it allows us to isotope a connected, stabilized spoke surface $S\subset W$ back and forth between small neighborhoods of dual spoke graphs lying on opposite sides of $S$ in $W$.  This kind of isotopy will play an essential role in the final doppelg\"{a}nger construction.

\end{rmk}

\begin{definition}

Let $\X$ be a disjoint union of stabilized spoke graphs embedded in $W$.  Then two components $X_1$ and $X_2$ of $\X$ are said to be:

\begin{itemize}

\item {\em $\phi$-adjacent} if there is a subarc $\beta\subset S^1$ such that the $\beta\times\{0\}\subset W$ meets $\X$ only in the central vertices of $X_1$ and $X_2$, on its endpoints.  In this case we call $\beta\times\{0\}$ the {\em spanning arc of $\phi$-adjacency}.

\item {\em $\theta$-adjacent} if the closure $A$ of some component of $\overline{\partial W-N(\X )}$ meets $N(X_1,W)$ in one boundary component and $N(X_2,W)$ in the other.  In this case $A$ is said to be a {\em spanning annulus of $\theta$-adjacency}.

\end{itemize}

\end{definition}

\begin{definition}

Let $\X$ be a disjoint union of stabilized spoke graphs whose components are ordered $X_1,\ldots X_n$ so that $X_i$ is $\phi$-adjacent to $X_{i+1}$, and let $\alpha_i$ be the spanning arc of $\phi$-adjacency for $X_i$, $X_{i+1}$, $1\leq i<n$. Then $\alpha=\alpha_1\cup\cdots\cup \alpha_n$ is said to be the {\em binding arc} of $\X$ with respect to the given ordering of its components. 

Suppose further that $X_i$ is $\theta$-adjacent to $X_{i+1}$, and let $A_i$ be a spanning annulus of $\theta$-adjacency for $X_i$, $X_{i+1}$, $1\leq i<n$.  Then $\A=A_1\cup\cdots \cup A_{n-1}$ is said to form an {\em adjacency chain} for $\X$ with respect to the given ordering of its components.

\end{definition}

\begin{definition}

Let $\X$ be a disjoint union of stabilized spoke graphs (possibly with detached longitudes), and let $X$ be a connected stabilized spoke graph (possibly with detached longitudes) obtained from $\X$ by rotating each component of $\X$ in the $\phi$-direction so that all of their central vertices coincide at a single vertex $x_0$.  Then $\X$ is said to be a {\em decomposition} of $X$.  If $d(x_0,v)<\epsilon$, for every central vertex $v$ occurring in a component of $\X$, where $d:W\times W\rightarrow \mathbb{R}$ is the flat metric, then $\X$ is said to be an {\em $\epsilon$-small} decomposition of $X$.

\end{definition}

\begin{rmk}

If $\X$ is an $\epsilon$-small decomposition of $X$, and $\epsilon<\pi/2$, then there are exactly two orderings $X_1,\ldots , X_n$ of the components of $\X$ with the property that $X_i$ is $\phi$-adjacent to $X_{i+1}$ for all $1\leq i<n$, and such that the binding arc $\alpha$ has length less than $2\epsilon$ (and these two orderings are just the reverse of one another).  If, moreover, $X_i$ is $\theta$-adjacent to $X_{i+1}$ for all $1\leq i<n$ for one (and therefore both) of these orderings, we shall say that $\X$ is a {\em good} $\epsilon$-small decomposition of $X$.

\end{rmk}

\begin{obs}

Let $\X$ be an $\epsilon$-small decomposition of $X$, let $\alpha$ be the binding arc of $\X$.  Then $Fr(\X\cup \alpha)$ is isotopic to $Fr(X)$, and $E(\X,W)$ is obtained from $E(X,W)\cong E(\X\cup \alpha,W)$ via $2$-handle attachments to $Fr(\X\cup \alpha, W)$ along meridians of $\alpha$.

\end{obs}

\begin{definition}

Let $X_1$ and $X_2$ be a pair of components in $\X$ which are $\phi$- and $\theta$-adjacent, with spanning arc of $\phi$-adjacency $\alpha$ and spanning annulus of $\theta$-adjacency $A$.  Let $D$ be disk embedded in $W$ such that $\partial D=\alpha\cup e_1\cup\beta\cup e_2$, where $e_i$ is the radial edge of $X_i$ nearest to $A$, $i=1,2$, $\beta\subset \partial W$ is an arc joining $e_1$ to $e_2$ which spans $A$, and $\mathring{D}\cap \X=\emptyset$.  Then $D$ is said to be a {\em spanning disk} of $X_1$ and $X_2$.  See Figure 3.

\end{definition}

\begin{figure}
    \centering
    \includegraphics[width=0.8\textwidth]{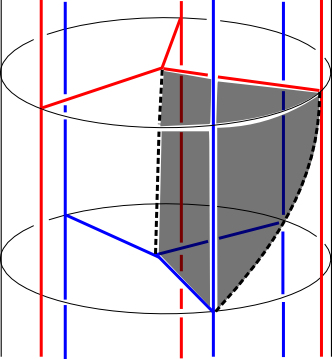}
    \caption{A spanning disk}
\end{figure}

\begin{definition}

Let $h_{\epsilon}:W\rightarrow W$ be the dilation $(\phi,r,\theta)\mapsto (\phi,(1-\epsilon)r,\theta)$.  Let $X$ be a stabilized spoke graph, let $\mathcal{L}$ be a union of longitudinal edges of $X$, let $\E$ be the union of those radial edges of $X$ which meet $\mathcal{L}$, and let $\mathcal{S}$ be the union of stabilizing arcs attached to $\E$.  Then the graph $X'$ obtained by removing $\E\cup \mathcal{L}\cup \mathcal{S}$ from $X$ and attaching $h(\E\cup \mathcal{L}\cup \mathcal{S})$ is said to be a {\em spoke graph obtained by $\epsilon$-small detachments of the longitudes $h(\mathcal{L})$}.  

\end{definition}

\begin{prop}

Suppose $X$ and $X'$ are dual stabilized spoke graphs embedded in $W$, that $\epsilon<\pi/2$, and that $d(X,X')>\epsilon$ (as usual $d$ is the flat metric).  Suppose $\X$ is an $\epsilon/8$-small, good decomposition of $X$, and that $\A=A_1\cup\cdots \cup A_k$ is an adjacency chain of annuli for $\X$.  Let $Y$ be a spoke subgraph of $X'$ which does not meet $\A$, and suppose $\Y$ is obtained from an $\epsilon/8$-small longitudinal detachments $Y$ followed by an $\epsilon/8$-small decomposition (which need not be ``good'').  Let $\A'$ be the subset of $E(\X\cup \Y,\partial W)$ consisting of those annuli which meet $Fr(\X)$ on one boundary component, and $Fr(\Y)$ on the other.  Then $E(\X\cup\Y,W)$ is a generalized compression body $V$ such that $\partial_vV=\A'$ and $\partial_-V=Fr(\Y)$.

\end{prop}

\begin{proof}

We have chosen the various stabilizations and detachments small enough to ensure that they can all be carried simultaneously without creating new intersections. As in the proof of Lemma 2.9, we double $W$ to obtain $S^1\times S^2$, let $X_d$ denote the double of $X$, and let $X_d'$ be the double of $X'$.  Then if $Y'$ is the graph obtained by detaching some of the longitudes of $Y$, it remains isotopic to a subgraph of $X_d'$, which is a spine of $E(X_d,S^1\times S^2)$.  Thus $E(X_d\cup Y',S^1\times S^2)$ is a compression body with negative boundary $\partial N(Y')$.  

If $\Y$ is any $\epsilon/8$-small decomposition of $Y'$, then $E(X_d\cup \Y, S^1\times S^2)$ is also compression body because Observation 2.15 tells us that it is obtained from $E(X_d\cup Y',S^1\times S^2)$ via $2$-handle attachments along $\partial N(Y',W)=\partial_-E(X_d\cup Y',S^1\times S^2)$.  The annuli $\A'$ are the spanning annuli in the collection $(\partial W)\cap E(X\cup \Y,S^1\times S^2)$ of incompressible annuli properly embedded in $E(X_d\cup \Y, S^1\times S^2)$, and thus form the vertical boundary of the generalized compression body $V'=E(X\cup \Y, W)$ cut off by $\partial W\cap E(X\cup\Y,S^1\times S^2)$.  

Order the components of $\X=X_1\cup\cdots \cup X_{k+1}$ so that $A_i$ meets $N(X_i,W)$ and $N(X_{i+1},W)$ for all $1\leq i\leq k$.  The hypothesis that $Y$ does not meet $\A=A_1\cup\cdots\cup A_k$ implies the existence of a spanning disk $D_i$ for each pair $X_i$, $X_{i+1}$ which is disjoint from $\Y$.  If $\alpha_i$ is the spanning arc of $\phi$-adjacency for $X_i$, $X_{i+1}$, then the meridian disk $D_i'$ of $\alpha_i$ in $N(\X\cup\alpha_1\cup\cdots \cup \alpha_k\cup \Y,W)$ meets the disk $D_i\cap E(\X\cup\alpha_1\cup\cdots \cup \alpha_k\cup \Y,W)$ in a single point.  Thus by Proposition 1.7, $E(\X\cup \Y,W)$, which is obtained by attaching $N(D_i')$ to $E(\X\cup\alpha_1\cup\cdots \cup \alpha_k\cup \Y,W)$ for each $1\leq i\leq k$, is also a generalized compression body.

\end{proof}

\section{The Doppelg\"{a}nger}

\begin{conv}
Throughout this section, $M$ is a compact orientable $3$-manifold, $\G$ is a generalized Heegaard surface of $M$, $T$ is a separating essential torus properly embedded in $M$, and $E(\G,M)=M_1\cup M_2$.
\end{conv}





\begin{definition}

$\G$ and $T$ are said to be {\em well-configured} with respect to $M_1$ if the following conditions hold:

\begin{enumerate}

\item $\G\cap T$ consists only of simple closed curves which are essential in $T$ and $\G$,
\item Each component of $\G\cap M_1$ separates $M_1$,
\item For each component $V$ of $E(\G,M)$, $T\cap V$ consists only of annuli which are spanning or horizontal.

\end{enumerate}

\end{definition}


\begin{conv}

For the remainder of the section, we assume that $\G$ and $T$ are well-configured with respect to $M_1$.

\end{conv}

\begin{obs}
Let $\Hf=E(\G,M)$, which is a disjoint union of compression bodies, let $\A=Fr(T\cap \Hf,\Hf)$, $\A^1=\A\cap M_1$, and $\A^2=\A\cap M_2$.  Then conditions (1) and (3) of Definition 3.3, together with Proposition 1.10, imply that $\V=E(T\cap \Hf,\Hf)$ is a generalized compression body satisfying $\partial_v\V=\A_s$, where $\A_s$ is the union of spanning annuli in $\A$.  Moreover, if $\A_h=\A-\A_s$ is the subset of horizontal annuli in $\A$, and $\A_h^i=\A^i\cap A_h$, $i=1,2$, then there is a primitive disk set $\D$ for $\V$ with respect to some ordering of $\A_h^1\cup\A_h^2$
\end{obs}

\begin{conv}
The notation of Observation 3.4 is fixed for the remainder of the section.  Moreover, we fix a choice of a primitive disk set $\D=D_1\cup\cdots\cup D_n$, which imposes the primitive orderings $\A_h^i=A_1^i\cup\cdots \cup A_n^i$, $i=1,2$.  Here it is understood that, for all $1\leq j\leq n$, $\A_j^1\cup \A_j^2$ is the frontier of a single component of $T\cap \Hf$.
\end{conv}

\begin{definition}

Let $V$ be a component of $E(T\cap \Hf,\Hf)$.  Let $\D^V=\D\cap V$, let $\A^V=\A\cap V$, $\A_s^V=\A_s\cap V=\partial_vV$, $\A_h^V=\A_h\cap V$, and let $\A_p^V$ consist of those components of $\A_h^V$ which are dual to some component of $\D^V$.

\end{definition}

\begin{lem}

For every component $V$ of $\V$, $\partial_+V-\A_p^V$ is connected.

\end{lem}

\begin{proof}

In the case $|\A_p^V|=0$ there is nothing to prove, so assume $|\A_p^V|>0$.  Order the components of $\D^V=D_1\cup\cdots \cup D_k$ so that $i<j$ if $D_i$ has lower index than $D_j$ with respect to the primitive ordering of $\D$.  Similarly, order $\A_p^V=A_1\cup\cdots \cup A_k$ so that $i<j$ if $A_i$ has lower index than $A_j$ with respect to the primitive ordering of $\A_h$ (so $D_i$ is dual $A_i$ for all $1\leq i\leq k$).  The fact that $\D$ is primitive implies that $\partial D_1$ meets $\A_p^V$ only in a single spanning arc of $A_1$, so that $A_1$ meets a single component of $\overline{\partial_+V-\A_p^V}$.  Likewise, $\partial D_2$ is disjoint from $\A_p^V-(A_1\cup A_2)$ and meets $A_2$ only in a single spanning arc.  Since $\partial D_2$ does not change the component of $\partial_+V-\A_p^V$ on which it lies when it passes through $A_1$, it follows that $A_2$ also meets the same component of $\overline{\partial_+V-A_p^V}$ on each side.  Continuing in this way for the remaining components of $\A_p^V$, we see that every component of $\A_p^V$ meets a single component of $\overline{\partial_+V-\A_p^V}$.  Since $\partial_+V$ is connected, this implies that $\partial_+V-\A_p^V$ is also connected.

\end{proof}

\begin{definition}

Let $V$ be a component of $\V$, and index the annuli $\A_h^V=A_1\cup\cdots\cup A_m$ so that $i<j$ implies that $A_i$ has lower index than $A_j$ with respect to the primitive ordering on $\A_h$.  A set $\mathbf{A}=\{A_{i_1},\ldots , A_{i_k}\}$ of components of $\A_h^V-\A_p^V$ is said to be {\em connective} if $(\partial_+V-\A_h^V)\cup A_{i_1}\cup\cdots\cup A_{i_k}$ is connected.  Moreover, $\mathbf{A}$ is {\em minimal} if (1) $|\mathbf{A}|$ is minimal among all connective sets of $V$ and, (2) for every other connective set of annuli $\mathbf{A}'=\{A_{j_1},\ldots ,A_{j_k}\}$ satisfying $|\mathbf{A}|=|\mathbf{A}'|$, $i_l\leq j_l$ for all $1\leq l\leq k$.  Note that conditions (1) and (2) define a unique minimal connective set with respect to any given ordering.

\end{definition}

\begin{conv}

For the remainder of the section, let $W\cong S^1\times D^2$ be a solid torus, parameterized as in Convention 2.1.  Furthermore, let $T_i=Fr(T)\cap M_i$, $i=1,2$ and let $h:T_1\rightarrow \partial W$ be a homeomorphism such that each component of $h(\G\cap T_1)$ is a longitude of the form $S^1\times\{x\}$.  Let $\pi:T_2\rightarrow T_1$ be the projection which collapses $T_2$ onto $T_1$ along the $I$-fibers of $N(T,M)$ (we assume that $\pi(T_2\cap \G)=T_1\cap \G$).  We let $M'=W\cup_{h\circ \pi}M_2$ for the remainder of the section.

\end{conv}

\begin{obs}

If a component $V$ of $\V\cap M_1$ meets $T_1$ at all, then $\partial_+V$ must meet $T_1$, since the annuli of $\A^V=T_1\cap V$ are all either horizontal or spanning.  However, it is possible that $\A^V$ consists entirely of spanning annuli, so that $\partial_-V$ does not meet $T_1$, and this is a case which requires special treatment at certain points in our construction.

\end{obs}




\begin{definition}

For each component $V$ of $\V\cap M_1$ which meets $T_1$, let $\B^V$ denote $\overline{\partial W-h(\A^V)}$ which is a union of annuli.  

\end{definition}

\begin{lem}

For each component $V$ of $\V\cap M_1$, and each component $B$ of $\B^V$, both components of $h^{-1}(\partial B)$ lie on the same component of $\overline{\partial V-\A^V}$.  

\end{lem}

\begin{proof}

If the curves of $h^{-1}(\partial B)$ lie on distinct components $F_1,F_2$ of $\partial V-\A^V$, then we can construct an embedded curve in $M_1$ which is the union of a spanning arc $\alpha$ of $h^{-1}(B)$ and an arc $\beta$ properly embedded in $V$ with $\partial \alpha=\partial\beta$.  This curve would then intersect the surface $F_1$ in a single point, which contradicts the assumption that each component of $\G\cap M_1$ (and hence each component of $\overline{\partial V-\A^V}$), is separating in $M_1$.

\end{proof}




\begin{definition}

Let $V$ be a component of $\V\cap M_1$, and let $F$ be a component of $\overline{\partial V-\A^V}$ that meets $T_1$.  Let $\B_F$ be the union of those components $B$ of $\B^V$ such that $h^{-1}(\partial B)\subset F$.  If $X$ is a connected, stabilized spoke graph, possibly with detached longitudes, whose non-detached longitudinal edges are the core curves of $\B_F$, then $X$ is said to be a {\em doppelg\"{a}nger spoke graph} for $F$.

\end{definition}

\begin{definition}

Let $V$ be a component of $\V\cap M_1$, and let $\X\cup\Y$ be a spoke graph constructed as follows:

\begin{enumerate}

\item Let $\B_+^V$ be the union of those annuli in $\B^V$ which appear in $\B_F$ for some component $F$ of $\overline{\partial_+V-\A^V}$, let $X$ be a connected stabilized spoke graph whose longitudinal edges are the core curves of $\B_+^V$, and let $X'$ be its dual.  Suppose $d(X,X')=\epsilon<\pi/2$, where $d$ is the flat metric on $W$ as in Lemma 2.17.

\item Suppose $\mathbf{A}=\{A_{i_1},\ldots A_{i_k}\}$ is the minimal connective set of components of $\A_h^V-\A_p^V$ defined in Definition 3.8.  For some pair of components of $\overline{\partial_+V-\A^V}$, label them $F_j$ and $F_{j+1}$, $\partial A_{i_j}$ has one component in $\partial F_j$ and the other in $\partial F_{j+1}$.  Since $\mathbf{A}$ was chosen to be minimal, $F_j$ will be distinct from $F_l$ whenever $j\neq l$, (otherwise we could remove an element of $\mathbf{A}$ and still have a connective set).  On the other hand, since $\mathbf{A}$ is connective, $\overline{\partial_+V-\A^V}=F_1\cup\cdots \cup F_{k+1}$.  Thus there exists an $\epsilon/8$-small decomposition $\X$ of $X$ such that $\X=X_{F_1}\cup\cdots \cup X_{F_{k+1}}$, where $X_{F_j}$ is a doppelg\"{a}nger spoke graph of $F_j$, and moreover $\X$ can be chosen so that $C_1\cup\cdots \cup C_k$ forms an adjacency chain for $\X$, where here $C_j$ denotes the component of $\partial W-\mathring{N}(X,W)$ which contains $h(A_{i_j})$ (see Definition 2.12). 

\item Let $\B_-^V=\B^V-\B_+^V$, which is the union of those annuli in $\B^V$ which appear in $\B_F$ for some component $F$ of $\partial_-V$ that meets $T_1$.  Every component of $\partial W-X$ contains at most one component of $\B_-^V$, for if two components of $\B_-^V$ both lied in the same component of $\partial W-X$, this would imply the existence of a component of $\A^V$ whose boundary components both lie in $\partial_-V$, contrary to part (3) of Definition 3.2.  Thus we may assume that the core curves of $\B_-^V$ form a subset of the longitudinal edges of the dual $X'$ of $X$.

\item  Define the {\em prohibited} longitudinal edges of $X'$ to be those which lie in the same component of $\partial W-X$ as a component of $h(\A_p^V)\cup h(A_{i_1})\cup\cdots \cup h(A_{i_k})$ (there is at most one component of this set lying inside each component of $\partial W-X$).  A subgraph $Y'$ of $X'$ is said to be {\em admissable} if it possesses every core curve of $\B_-^V$ as a longitudinal edge, but no prohibited longitudinal edges.

\item It is possible that $\B_-^V=\emptyset$, in which case we set $\Y=\emptyset$.  Otherwise, let $F_1',\ldots ,F_l'$ be the components of $\partial_-V$ which meet $T_1$.  Let $Y'$ be an admissible subgraph of $X'$, and let $Y$ be obtained from $Y'$ via an $\epsilon/8$-small detachment of those longitudes of $Y'$ which are not core curves of $\B_-^V$.   Finally, let $\Y=Y_{F_1'}\cup\cdots \cup Y_{F_l'}$ be an $\epsilon/8$-small decomposition of $Y$, where $Y_{F_j'}$ is a doppelg\"{a}nger spoke graph for $F_j'$, $1\leq j\leq l$.

\end{enumerate}

Then $\X\cup \Y=X_{F_1}\cup\cdots \cup X_{F_k}\cup Y_{F_1'}\cup\cdots\cup Y_{F_l'}$ is said to be a {\em doppelg\"{a}nger spoke graph} of $V$, and it is said to be {\em perfect} if $Fr(X_{F_j})\cong F_j$ and $Fr(Y_{F_r'})\cong F_r'$ for all $1\leq j\leq k$, $1\leq r\leq l$.

\end{definition}

\begin{obs}

The construction of Definition 3.14 was tailored to the hypotheses of Proposition 2.18.  It implies that if $\X\cup \Y$ is a doppelg\"{a}nger spoke graph associated with $V$, then $U=E(\X\cup \Y,W)$ is a generalized compression body satisfying $\partial_-U=Fr(\Y,W)$.  $\partial_+U$ is then the union of $Fr(\X,W)$ with those components of $\overline{\partial W-N(\X\cup \Y, W)}$ whose boundary components both lie in $Fr(\X,W)$.  Moreover, the minimality of $\mathbf{A}$ stipulated in part (2) of Definition 3.14 allow us to deduce the existence of a primitive disk set $\E^U$ in $U$ which will serve as a substitute for the disk set $\D^V$ of $V$, as the following lemma shows.

\end{obs}

\begin{figure}
    \centering
    \includegraphics[width=0.8\textwidth]{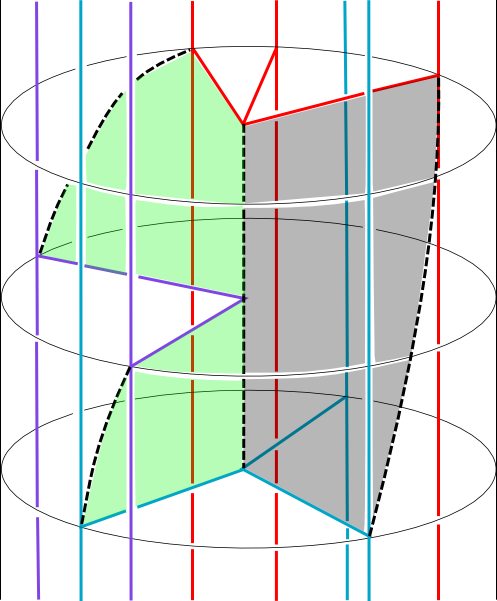}
    \caption{A disk with flaps}
\end{figure}

\begin{lem}

Let $V$ be a component of $\V\cap M_1$, let $\X\cup\Y$ be a doppelg\"{a}nger spoke graph of $V$, and let $U=E(\X\cup\Y,W)$.  Suppose $\A_h^V=A_1\cup\cdots\cup A_m$ is ordered as in Definition 3.8, and that $\A_p^V=A_{p_1}\cup\cdots\cup A_{p_q}$.  Then there is an ordered, disjoint collection of disks $\E^U=E_{p_1}\cup\cdots\cup E_{p_q}$ properly embedded in $U$ with the following properties:

\begin{enumerate}

\item $\E^U\cap \partial_vU=\emptyset$,
\item $E_{p_j}\cap h(A_{p_j})$ is a single spanning arc for all $1\leq j\leq q$,
\item $E_{p_j}\cap h(A_l)=\emptyset$ for all $1\leq j\leq q$, $p_j<l\leq m$.

\end{enumerate}

\end{lem}

\begin{proof}

For each component $A_l$ of $\A_h^V$, let $C_l$ denote the closure of the component of $\partial W-\mathring{N}(\X\cup \Y)$ which contains $h(A_l)$.  Let $\mathbf{A}=\{A_{i_1},\ldots , A_{i_k}\}$ be the minimal connective set for $\A_h^V$, and order the components of $\X=X_1\cup\cdots \cup X_{k+1}$ so that the components of $\partial C_{i_j}$ lie in $Fr(X_j)$ and $Fr(X_{j+1})$ for all $1\leq j\leq k$.  In the next two paragraphs we describe the components $E_{p_j}$ of $\E^V$.

Suppose first that $\partial C_{p_j}$ lies on a single component $Fr(X_r)$ of $Fr(\X)$, where $X_r$ has central vertex $v=(\phi_0,0,0)$.  In the notation of part (4) of Definition 3.13, $Y'$ must be disjoint from $C_{p_j}$, since the longitude of $X'$ that lies in $C_{p_j}$ is prohibited ($C_{p_j}$ contains a component of $A_p^V$ by definition). Hence $\Y$ is disjoint from the component of $(\{\phi_0\}\times D^2)\cap E(\X,W)$ that meets $C_{p_j}$.  We will let $E_{p_j}$ denote this component, which is a disk properly embedded $E(\X\cup\Y,W)$.  Call this kind of disk a {\em simple disk}.

If the components of $\partial C_{p_j}$ lie on distinct components $Fr(X_r)$ and $Fr(X_s)$ of $Fr(\X)$, $r<s$, then the disk we need to construct is a bit more complex.  First let $G$ be the union of the spanning disks (recall Definition 2.16) for the annuli $C_{p_r}, C_{p_{r+1}},\ldots C_{p_{s-1}}$, and let $D$ be the spanning disk of $C_{p_j}$.  Then let $E_{p_j}$ be the disk $(D\cup G)\cap E(\X,W)$, which again is disjoint from $\Y$ by part (4) of Definition 3.14, and is thus properly embedded in $E(\X\cup\Y,W)$.  Call this kind of disk a {\em disk with flaps}, see Figure 4.

A simple disk is disjoint from all the other disks defined above, and any disks with flaps which overlap one another can also be isotoped slightly to be disjoint from one another.  The resulting collection of disks is our collection $\E^U$.  Obviously $\E^U$ satisfies conclusions (1) and (2).  Conclusion (3) is also clear in the case that $E_{p_j}$ is simple, since in fact the only annulus of the form $h(A_l)$ which it meets is $h(A_{p_j})$.  

In the case that $E_{p_j}$ has flaps, condition (2) on minimality for $\mathbf{A}$ given in Definition 3.8 ensures that the index of $E_{p_j}$ is higher than that of every element element $A_l\in\mathbf{A}$ such that $h(A_l)\cap E_{p_j}\neq \emptyset$.  Otherwise we could replace $A_l$ with $A_{p_j}$ in $\mathbf{A}$ to obtain another connective set which violates the minimality of $\mathbf{A}$.  Since $E_{p_j}$ meets no other annuli of the form $h(A_l)$ except for $h(A_{p_j})$, the final condition is satisfied.

\end{proof}

\begin{prop}

For every component $V$ of $\V\cap M_1$, there is a perfect doppelg\"{a}nger spoke graph embedded in $W$.

\end{prop}

\begin{proof}

For any doppelg\"{a}nger spoke graph $X_F$ of a component $F$ of $\overline{\partial V-\A^V}$, the surface $Fr(X_F,W)$ will have the same number of boundary components as $F$.  Moreover, the genus of $Fr(X_F,W)$ is the same as the total number of stabilizing arcs and detached longitudes that occur in $X_F$.  In particular, if $X_F$ is unstabilized and has no detached longitudes, then $Fr(X_F,W)$ will be planar.  Thus a doppelg\"{a}nger spoke graph $X_F$ can always be found which satisfies $Fr(X_F,W)\cong F$ after attaching a sufficient number of stabilizing arcs and/or detached longitudes.  

In part (1) of Definition 3.14, we have the flexibility to stabilize $X$ as often as we need, with stabilizing arcs based on radial edges of our choosing.  This allows us to choose the number of stabilizing arcs that will eventually occur in the components of the spoke graph $\X$ defined in part (2) of Definition 3.14.  It follows from this and the previous paragraph that we may choose $\X$ so that each of its components $X_{F_j}$ satisfies $Fr(X_{F_j},W)\cong F_j$.  

As noted in Remark 2.7, the Morse condition of Definition 2.6 grants us enough flexibility to choose the radial edges on which the stabilizing arcs of $X'$ (the dual of $X$) will be based, and this in turn allows us to control the component of $\Y$ on which they will eventually occur in part (5) of Definition 3.14. Likewise, we can choose the components of $\Y$ on which the detached longitudes of $Y$ shall occur after the decomposition described in part (5) of Definition 3.14.  

So, similar to the case with $\X$, we may distribute detached longitudes and stabilizing arcs among the components of $\Y$ however we please.  But there is an important difference: The total number of stabilizing arcs and detached longitudes that can occur in $\Y$ is bounded above by $s+a$, where $s$ denotes the number of stabilizing arcs that occur in $X$, and $a$ denotes the maximal number of detached longitudes that can occur on an admissible subgraph of $X'$ (as defined in part (4) of Definition 3.14).  Therefore, to complete the proof we must show that a total of $s+a$ stabilizing arcs and detached longitudes is always sufficient to create a spoke graph $\Y$ which satisfies the equation $Fr(Y_{F_j'},W)\cong F_j'$ for each of its components $Y_{F_j'}$.

As in part (5) of Definition 3.14, let $\F'=F_1'\cup\cdots \cup F_l'$ be the union of those components of $\partial_-V$ which meet $T_1$.  By the first paragraph of this proof, the total number of stabilizing arcs and detached longitudes necessary to ensure that $Fr(Y_{F_j'}, W)\cong F_j'$ for all $1\leq j\leq l$ is equal to $\sum g(F_j')$, where $g(F_j')$ is the genus of $F_j'$.  Since $g(F)=1-\frac{\chi (F)+|\partial F|}{2}$ for any connected, compact surface $F$, we obtain 

\[
\displaystyle\sum g(F_j')=|\F'|-\frac{\chi (\F')+|\A_s^V|}{2}.\tag{1}
\]

The fact that this quantity is less than $s+a$ is ultimately derived from the inequality 

\[
\chi(\partial_+V)-\chi(\partial_-V)\leq -2(|\A_p^V|+|\partial_-V|-1).\tag{2}
\]

The truth of (2) can be seen as follows: $V'=E(\D^V,V)$ is a generalized compression body with the same negative boundary as $V$.  Furthermore, $\partial_+V'$ is connected since $\partial_+V-\partial \D^V$ is connected, as can be seen using essentially the same proof as that of Lemma 3.7.  Now $\partial_-V'$ is obtained from $\partial_+V'$ via surgeries along disks, and there must be at least $|\partial_-V'|-1$ such surgeries since $\partial_+V'$ is connected.  Thus $\chi(\partial_+V')-\chi(\partial_-V')\leq -2(|\partial_-V'|-1)$.  Inequality (2) now follows from the fact that $\partial_-V'=\partial_-V$, and the fact that $\chi(\partial_+V')=\chi(\partial_+V)+2|\D^V|=\chi(\partial_+V)+2|\A_p^V|$.

Since no component of $\partial_-V$ is a disk or sphere, and $\F'\subset \partial_-V$, $\chi (\partial_-V)\leq \chi(\F')$.  Thus from (2) we easily obtain the analogue $\chi(\partial_+V)-\chi(\F')\leq -2(|\A_p^V|+|\F'|-1)$.  In conjunction with equation (1), we obtain

\[
\displaystyle\sum g(F_j')\leq 1-|\A_p^V|-\frac{\chi (\partial_+V)+|\A_s^V|}{2}\tag{3}
\]

Our choice of $\X$ has ensured that $\chi (\partial_+V)=\chi (Fr(\X,W))$.  We then compute $\chi (Fr(\X,W))=-|\partial Fr(\X,W)|-2s+2|Fr(\X,W)|=-2|\A_h^V|-|\A_s^V|-2s+2|\X|$, and so deduce

\[
\displaystyle\sum g(F_j')\leq 1-|\A_p^V|+|\A_h^V|+s+|\X|\tag{4}
\]

The detached longitudes of $Y$, as described in part (5) of Definition 3.14, all come from core curves of the annuli of $\partial W-X$ which contain a component of $h(\A_h^V)$.  On the other hand, the number of prohibited annuli (part (4) of Definition 3.14) is equal to $|\A_p^V|+|\mathbf{A}|$ (where $\mathbf{A}$ denotes the set of annuli defined in part (2) of Definition 3.14).  Hence there will be at most $a=|\A_h^V|-|A_p^V|-|\mathbf{A}|$ detached longitudes which may occur in $\Y$.  However, $|\mathbf{A}|=|\X|-1$ by the minimality of $|\mathbf{A}|$.  Plugging this into inequality (4) yields 

\[
\displaystyle\sum g(F_j')\leq s+a.\tag{5}
\]

The lemma now follows.

\end{proof}

\begin{definition}
Let $V$ be a component of $\V\cap M_1$, let $\G^V$ denote the union of those components $G$ of $\G\cap M_1$ such that $N(G,M_1)\cap V\neq \emptyset$.  Let $\X\cup \Y$ be a perfect doppelg\"{a}nger spoke graph for $V$, and let $\Q^V$ be the result of isotoping $Fr(\X\cup\Y,W)$ so that $\partial Fr(\X\cup\Y,W)=h(\partial \G^V)$, via an isotopy supported in $N(\B^V,W)$ (here $\B^V=\overline{\partial V-h(\A^V)}$ is the union of annuli described in Definition 3.11).  Then $\Q^V$ is called the {\em doppelg\"{a}nger surface} of $V$, and the closure of the component of $W-\Q^V$ which does {\em not} contain $\X\cup\Y$ is the {\em doppelg\"{a}nger chamber} of $V$.
\end{definition}

\begin{rmk}

There is little difference between the doppelg\"{a}nger chamber $U$ of $V$ and $E(\X\cup\Y,W)$, outside of the fact that the boundary components $\partial_+U\cup\partial_-U=\Q^V$ line up well with respect to our gluing map.  In particular, Observation 3.15 and Lemma 3.16 apply in the case that $U$ is a doppelg\"{a}nger chamber.

\end{rmk}

\begin{figure}
    \centering
    \includegraphics[width=0.95\textwidth]{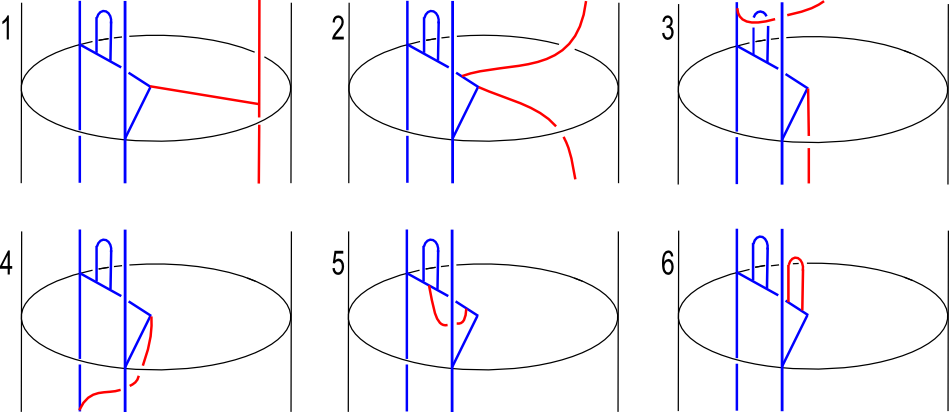}
    \caption{Edge slides that turn a detached longitude into a stabilized arc}
\end{figure}

\begin{lem}

Let $X$ be a connected, unstabilized spoke graph and let $X'$ be its dual.  Suppose $Y$ is obtained from $X$ by attaching a total of $n$ stabilizing arcs and detached longitudes, and that $Y'$ is obtained from $X'$ by attaching a total of $n$ stabilizing arcs and detached longitudes, in any fashion.  Then $Fr(Y,W)$ is properly isotopic to $Fr(Y',W)$ in $W$.

\end{lem}

\begin{proof}

If $Y$ has detached longitudes, then there is a sequence of edge slides which change $Y$ into a connected stabilized spoke graph $\tilde{Y}$ without detached longitudes (see Figure 5), and these correspond to an isotopy of $Fr(Y,W)$ to $Fr(\tilde{Y},W)$, one which we can choose to be supported outside of a small open collar of $\partial W$. 

We may then isotope $Fr(\tilde{Y},W)$ onto $Fr(\tilde{Y}',W)$ via a ``sweepout'' isotopy, where $\tilde{Y}'$ is the dual of $\tilde{Y}$.  We will call this the {\em sweepout stage} of the isotopy between the surfaces.  Note that it is the only stage of the isotopy during which the boundary of our surface is not fixed.

Finally, since $\tilde{Y}'$ is obtained from $X'$ by attaching $n$ stabilizing arcs, using (a reversed version of) the same kind of isotopy described in the first paragraph, we may slide the stabilizing arcs of $\tilde{Y}'$ along $X'$ so that they coincide with the stabilizing arcs and detached longitudes of $Y'$.  This corresponds to an isotopy of $Fr(\tilde{Y}',W)$ onto $Fr(Y',W)$ which fixes $\partial Fr(\tilde{Y}',W)$.  Composing these istopies yields the result.

\end{proof}

\begin{definition}
Let $V$ and $V'$ be a pair of components of $\V\cap M_1$, and let $\G^V$, $\G^{V'}$ be defined as in Definition 3.18.  Then $V$ and $V'$ are said to be {\em adjacent} if $\G^V\cap\G^{V'}\neq \emptyset$.  
\end{definition}

\begin{lem}

Suppose $V$ and $V'$ are an adjacent pair of components of $\V\cap M_1$.  Then $\G^V\cap\G^{V'}$ consists of a single component, call it $G$.  Moreover, if $Q^V$ and $Q^{V'}$ are the component of $\Q^V$ and $\Q^{V'}$, respectively, which satisfy $h(\partial G)=\partial Q^V=\partial Q^{V'}$, then $Q^V$ is isotopic to $Q^{V'}$ in $W$ via an isotopy which fixes $\partial Q^V$.

\end{lem}

\begin{proof}
If $\G^V\cap G^{V'}$ had at least two components $G_1$ and $G_2$, then we could embed a simple closed curve $\omega\subset M_1$ which meets each of $V$, $V'$, $N(G_1,M_1)$ and $N(G_2,M_1)$ in a single arc.  Since $\omega$ meets $G_1$ in a single point, this contradicts our assumption that each component of $\G\cap M_1$ is separating in $M_1$.  Thus there is exactly one component, $G$, of $\G^V\cap G^{V'}$.  

Let $F=N(G,W)\cap V$ and $F'=N(G,W)\cap V'$.  Then $F$ is a component of $\overline{\partial V-\A^V}$, $F'$ is a component of $\overline{\partial V'-\A^{V'}}$, and we let $X_F$, $X_{F'}$ denote the corresponding perfect doppelg\"{a}nger spoke graphs which give rise to $Q^V$ and $Q^{V'}$.  Setting $X_F=Y$ and $X_{F'}=Y'$ in Lemma 3.20, we see that there is an isotopy of $Fr(X_F,W)$ onto $Fr(X_{F'},W)$.  Outside of a small open collar $C$ of $\partial W$, the isotopy we require from $Q^V$ to $Q^{V'}$ is identical to the isotopy described in the proof of Lemma 3.20.  The only difference is that, during the sweepout stage, $Q^V\cap C$ will be taken onto $Q^{V'}\cap C$ via an isotopy which leaves $\partial Q^V$ fixed, as shown in Fig. ??
\end{proof}

\begin{thm}

If the generalized Heegaard surface $\G\subset M$ amalgamates to a minimal genus Heegaard surface of $M$, then $g(M')\leq g(M)$ (here $M'$ is the manifold obtained by gluing $W$ to $M_2$ as described in Convention 3.9).

\end{thm}

\begin{proof}

Our strategy is to construct a surface $\Q\subset W$ so that $\Q\cup_{h\circ\pi}(\G\cap M_2)$ (see Convention 3.9) forms a generalized Heegaard surface of $M'$, one which amalgamates to a Heegaard surface of at most the same genus as the Heegaard surface that $\G$ amalgamates to in $M$.  This will yield the desired result.

Our assumption that $T$ and $\G$ are well-configured does not eliminate the possibility that $T\cap\G=\emptyset$.  But in this case $T$ will be parallel to a component of $\G_-$, the thin part of $\G$, which implies the stronger conclusion $g(M_1)+g(M_2)=g(M)$.  So we assume $T\cap \G\neq \emptyset$.  

Let $V$ be a component of $\V\cap M_1$ which meets $T_1$, and let $\Q^V$ be its doppelg\"{a}nger surface.  By Lemma 3.22, for each component $V'$ which is adjacent to $V$, there is exactly one component $Q^V$ of $\Q^V$ which is isotopic to a component $Q^{V'}$ of $\Q^{V'}$.  Since this isotopy fixes $\partial Q^V$, it may be extended to an ambient isotopy of $W$ which fixes $\partial W$, and this ambient isotopy will push the remaining components of $\Q^V$ and the doppelg\"{a}nger chamber $U$ of $V$, into the complement of the doppelg\"{a}nger chamber $U'$ of $V'$ in $W$, thereby allowing $U'$ and the remaining components of $\Q^{V'}$ to be embedded in $W$ without meeting the deformed copies of $\Q^V$ and $U$.  Call this ambient isotopy a {\em flipping isotopy}.  Since any flipping isotopy leaves $\partial W$ fixed, every disk embedded in the original version of the doppelg\"{a}nger $U$ of $V$ is deformed to a disk which meets $h(\A^V)$ in the same way  as the original.  Thus the deformed copy of $U$ still satisfies Observation 3.15 and Lemma 3.16.

Thus, by performing and reversing flipping isotopies across the components of $\Q^V$, we may simultaneously embed deformed, but topologically equivalent, versions of the doppelg\"{a}nger surfaces and chambers of every component of $\V\cap M_1$ that is adjacent to $V$.  This process of embedding can then be repeated for each of the components of $\V\cap M_1$ that are adjacent to $V$, and then to the components of $\V\cap M_1$ which are adjacent to the components of $\V\cap M_1$ that are adjacent to $V$, and so on.  This process will eventually terminate with the desired surface $\Q$, and $E(\Q,W)$ will be the disjoint union of (harmlessly deformed versions of) the doppelg\"{a}nger chambers of those components of $\V\cap M_1$ which meet $T_1$.

Observation 3.15 then implies that $E(\Q,W)\cup (\V\cap M_2)$ is a generalized compression body.  Moreover, if $\E$ is the union of all the (deformed versions of) the disk sets $\E^U$ defined in Lemma 3.16, then Lemma 3.16 implies that $\E\cup (\D\cap M_2)$ admits an ordering which makes it a primitive disk set for the ordered union of annuli $h(\A_h^1)\cup \A_h^2$ (see Definition 1.4 and Convention 3.5) with respect to the map $h\circ \pi$ (see Convention 3.9).  Thus, if we let $\V_1'=E(\Q,W)\cup (\V\cap M_2)$, and $f=h\circ\pi|_{\A_h^2}$ (here $\A_h^2=\A_h\cap M_2$ as in Observation 3.14), then by Proposition 1.6, $\V_2'=\V_1'/f$ is also a generalized compression body.  Moreover, $f'=h\circ\pi |_{\A_s^2}$ (here $\A_s^2=\A_s\cap M_2$) satisfies the requirements of Observation 1.2, hence $\V'=\V_2'/f'$ is also a generalized compression body.  But if $\G'=\Q\cup_{h\circ\pi}(\G\cap M_2)$, then $\V'=E(\G',M')$, so in fact $\G'$ is a generalized Heegaard splitting of $M'$.  Moreover, since the doppelg\"{a}nger surface $\Q$ was constructed from {\em perfect} spoke graphs as per Definition 3.18, the components of $\G'$ will all have exactly the same genus as the corresponding components of $\G$.  It follows that $\G'$ will amalgamate to a Heegaard surface of genus no higher than the surface which results from amalgamating $\G$.

\end{proof}

\begin{rmk}

For the sake of clarity is worth remarking that in the proof above we cannot necessarily conclude that $\G'$ amalgamates to a surface of the {\em same} genus as the one which $\G$ amalgamates to, because $\G'$ will have no components which correspond to any components of $\G$ which lie entirely inside of $M_1$.  In the event that such components of $\G$ do exist, it is easy to verify that $\G'$ will amalgamate to a surface of strictly lower genus.  It is also perhaps worth pointing out that, since the core $c$ of $W$ can be embedded in $\Q\subset \G'$, we can stabilize $\G'$ once (if necessary) to obtain a generalized Heegaard splitting of $E(c,M')\cong M_2$ of genus at most $g(M)+1$.  This allows us to deduce the following corollary.

\end{rmk}

\begin{cor}

If $\G$ amalgamates to a minimal genus Heegaard surface of $M$, then $g(M_2)\leq g(M)+1$.

\end{cor}

\section{The Main Result}

\begin{thm}

If $K_1$ and $K_2$ are knots in $S^3$, then $t(K_1\# K_2)\geq \max\{t(K_1),t(K_2)\}$.

\end{thm}

\begin{proof}

The proposition is trivial if one of $K_1$ or $K_2$ is the unknot, so suppose that both are non-trivial knots.  Assume also that $\max\{t(K_1),t(K_2)\}=t(K_2)$.  Let $T$ be the ``swallow-follow'' torus in $E(K_1\# K_2)$ which swallows the $K_2$ summand and follows the $K_1$ summand (see Figure 6).  We apply Theorem 3.23 by setting $M=E(K_1\# K_2)$, noting that one component of $E(T,M)$ is homeomorphic to $E(K_1)$, which will correspond to $M_1$.  What needs to be shown is that the untelescoped minimal splitting $\G$ can be isotoped so that it meets $T$ only in essential simple closed curves, and such that each component of $\G\cap E(K_1)$ is separating.

By its definition as a swallow-follow torus, $T$ is isotopic to $A\cup B$, where $A$ is the decomposing annulus of the connected sum in $M=E(K_1\# K_2)$, and $B$ the sub-annulus of $\partial M-A$ which lies in the component of $M-A$ corresponding to $E(K_1)$.  By Proposition 1.21, $\G$ can be isotoped to intersect $A$ only in essential simple closed curves, and since each boundary component of $\G\cap E(K_1)$ is then a standard meridional curve of $\partial E(K_1)$, every component of $\G\cap E(K_1)$ is separating in $E(K_1)$ (otherwise we could obtain a non-separating surface in $S^3$).  The hypotheses of Theorem 3.23 (which assume Conventions 3.1, 3.3, and 3.9) can then be satisfied by isotoping $T$ sufficiently close to $A\cup B$.

Now $M_2$ is the component of $E(T,M)$ which is {\em not} homeomorphic to $E(K_1)$, but is instead homeomorphic to $E(L)$, where $L$ is the link in $S^3$ which has $K_2$ as one component, and a meridian $\mu$ of $K_2$ as its other component, and $T=\partial N(\mu)$ under this correspondence (see Figure 6).  Furthermore, the slope in which $\G$ has been made to intersect $T=\partial N(\mu )$ is the standard longitudinal slope determined by the meridian disk $\Delta\subset S^3$ with $\partial \Delta =\mu$ and $|\Delta\cap K_2|=1$.  Thus the slope of the trivial Dehn filling of $\partial N(\mu )=T$ which yields $E(K_2)$ meets each component of $T\cap \G$ exactly once, and Theorem 3.23 applies to $M'=E(K_2)$, yielding $$t(K_1\# K_2)=g(E(K_1\# K_2))-1\geq g(E(K_2))-1=t(K_2).$$

\end{proof}

\begin{figure}
    \centering
    \includegraphics[width=0.95\textwidth]{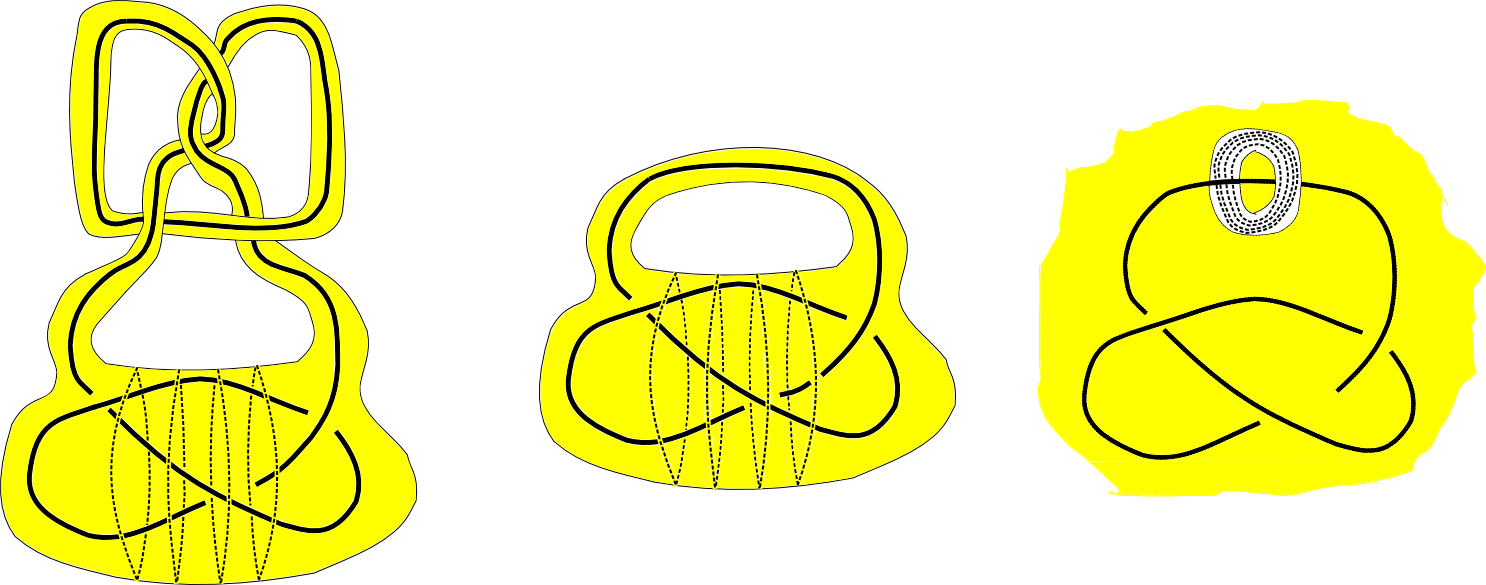}
    \caption{In the first diagram, $M_2$ (in yellow) is seen as situated in $E(K_1\#K_2,S^3)$, and $\G\cap T$ is indicated with dashed lines lying on the ``swallow-follow'' torus $T$.  In the second diagram, $M_2$ is reimbedded in $E(K_2,S^3)$, and in the final diagram we see $M_2$ and $\G\cap T$ as they look after inverting the image of $T$ under this reimbedding.}
\end{figure}

This proof also works if the knots $K_1$ and $K_2$ are embedded in homology spheres (or any pair of compact $3$-manifolds in which every closed embedded surface is separating).  In general, however, it is important to keep in mind the delicacy Theorem 3.23 (and Corollary 3.25), whose assumptions are encoded in Conventions 3.1, 3.3 and 3.9.  In particular, the assumption of Convention 3.3 that $T$ and $\G$ are well-configured cannot always be satisfied, as can be shown using straightforward examples in $S^1\times F$, where $F$ is a closed genus $g>1$ surface.  Thus Corollary 3.25 cannot be applied to prove the following plausible conjecture in the case $g=1$ in any obvious way.

\begin{conj}

Suppose $M$ is a compact $3$-manifold and $T$ is a separating, incompressible, orientable, genus $g$ surface properly embedded in $M$.  If $M_1$ and $M_2$ are the components of $E(T,M)$, then $g(M)\geq \max\{g(M_1),g(M_2)\}-g$.

\end{conj}

Similarly, the need for $T$ and $\G$ to be well-configured is what keeps us from applying Theorem 3.23 and Corollary 3.25 to prove the analogue of Theorem 4.1 for satellite knots.

Theorem 4.1 has some relation to the ``rank-genus conjecture'' for knot complements in $S^3$.  If we define $r(K)$ to be the minimal number of generators for $\pi_1(S^3-K)$, then the rank-genus conjecture states:

\begin{conj}

For all knots $K\subset S^3$, $r(K)=g(E(K,S^3))=t(K)+1$.

\end{conj}

Since a genus $g$ Heegaard splitting of a knot complement induces a $g$-generator presentation of $\pi_1(S^3-K)$, it is clear that $r(K)\leq t(K)+1$, but it remains unknown whether it is possible for this inequality to be strict.  The rank-genus conjecture for closed $3$-manifolds is known to be false \cite{boileau-zieschang}, even when restricted to the class of hyperbolic $3$-manifolds \cite{li}.  This suggests that the rank-genus conjecture also fails for knot complements, and a pair of knots in $S^3$ whose tunnel number degenerated enough to violate Theorem 4.1 would have given a counterexample, since the following analogue of Theorem 4.1 for rank is trivial (thanks to Richard Weidmann for pointing out the simple line of proof below).

\begin{prop}

$r(K_1\# K_2)\geq \max\{r(K_1),r(K_2)\}$ for any knots $K_1,K_2\subset S^3$.

\end{prop}

\begin{proof}

$\pi_1(E(K_1\# K_2))$ is an amalgamated free product $\pi_1(E(K_1))\ast_\mathbb{Z}\pi_1(E(K_2))$ which retracts onto each of its factors.

\end{proof}

The fact that Theorem 4.1 is true indicates that the class of knot pairs which experience high tunnel number degeneration is not a good place to look for counterexamples to the rank-genus conjecture after all.  It might even be seen as a small vote in favor of the possibility that the rank-genus conjecture is valid for knot complements in $S^3$.

\section*{Acknowledgments}

I would like to thank Maggy Tomova and Charlie Frohman for their support and guidance throughout my career.  I would like to thank Jesse Johnson for the same, as well as for many helpful conversations and observations about this paper in particular.

\end{document}